\documentclass[12pt, amsart]{article} 
\usepackage{amsfonts}
\usepackage{amssymb}
\usepackage{booktabs}
\usepackage{bm}
\usepackage{amsmath} 
\usepackage{amsfonts}
\usepackage{amssymb}
\usepackage{mathbbol}
\usepackage{hyperref}
\hypersetup{
	colorlinks=true,
	linkcolor=blue,
	filecolor=magenta,      
	urlcolor=cyan,
}
\usepackage[nameinlink]{cleveref}
\usepackage{latexsym}
\usepackage{multirow}
\usepackage{mathrsfs}
\usepackage{tikz}
\usepackage[all]{xy}
\usepackage{graphicx}
\usepackage{multicol}
\usepackage{amsmath}
\usepackage{relsize}
\usetikzlibrary{matrix,arrows,decorations.pathmorphing}
\usepackage{amsmath,amsthm}
\usepackage{float}
\floatstyle{boxed} 
\restylefloat{figure}
\usepackage[usenames,dvipsnames]{pstricks}
\usepackage{epsfig}
\usepackage{pst-grad} 
\usepackage{pst-plot} 
\usepackage[space]{grffile} 
\usepackage{etoolbox}
\usepackage{epstopdf}
\usepackage{amssymb}
\usepackage{latexsym}
\usepackage{color}

\usepackage{mathrsfs}
\usepackage{pst-grad} 
\usepackage{pst-plot}
\usepackage[latin1]{inputenc}
\usepackage{psfrag}
\usepackage{subfigure}
\usepackage{graphics}
\usepackage{graphicx}
\usepackage[usenames,dvipsnames]{pstricks}
\usepackage{pst-grad} 
\usepackage{pst-plot}
\usepackage{mathdots}
\usepackage[latin1]{inputenc}
\usepackage{psfrag}
\usepackage{graphics,graphicx}
\usepackage[english,  activeacute]{babel}
\usepackage{amssymb}
\usepackage{amsthm}
\usepackage{array}
\usepackage{pdflscape}
\usepackage{a4wide}
\usepackage{color, url}
\usepackage{float}
\theoremstyle{plain}
\newtheorem{theorem}{Theorem}
\newtheorem{proposition}{Proposition}
\newtheorem{corollary}[theorem]{Corollary}

\newtheorem{remark}[theorem]{Remark}
\theoremstyle{definition}
\newtheorem{definition}[theorem]{Definition}
\newtheorem{example}[theorem]{Example}
\newcommand{\x}{\mathbf{x}}
\newcommand{\Pib}{\mathbb{\Pi}}

\newcommand{\um}{\mathfrak{u}}
\newcommand{\ab}{\mathfrak{a}}
\newcommand{\Al}{\mathbb{\Sigma}}
\newcommand{\Sho}{\K\la\la \X\ran\ran}

\newcommand{\Se}{\mathcal{S}}
\newcommand{\La}{\mathscr{L}}
\newcommand{\C}{\mathscr{C}}
\newcommand{\Pa}{\mathcal{P}}
\newcommand{\pa}{\mathcal{P}}
\newcommand{\ra}{\rightarrow}
\newcommand{\raa}{\ran\ran}
\newcommand{\laa}{\langle\langle}
\newcommand{\A}{{\mathbb A}}

\newcommand{\N}{{\mathbb N}}
\newcommand{\Z}{{\mathbb Z}}
\newcommand{\X}{{\mathbb X}}
\newcommand{\Y}{{\mathbb Y}}
\newcommand{\Xpm}{{\mathbb X}_{\pm}}

\newcommand{\ord}{\mathrm {ord}}

\newcommand{\sh}{\sigma}

\newcommand{\la}{\langle}
\newcommand{\ran}{\rangle}
\newcommand{\RR}{\mathcal{R}}

\newcommand{\lam}{\lambda}
\newcommand{\om}{\omega}
\newcommand{\msA}{\mathscr{A}}

\newcommand{\spl}{\circ_s}

\newcommand{\kap}{{\boldsymbol \kappa}}
\newcommand{\free}{\mathscr{F}}

\newcommand{\K}{\mathbb{K}}
\newcommand{\Lau}{\mathscr{L}_{\infty}}
\providecommand{\keywords}[1]
{\small	
	\textbf{\textit{Keywords---}} #1}
\providecommand{\subjclass}[1]
{\small	\textbf{\textit{ Subject class---}} #1
}

\title{Abstract, keywords and references template}
\author{Miguel A. M\'endez\\
[-0.8ex]\small Universidad Yachay\\
	[-0.8ex]\small School of Mathematical\\[-0.8ex]\small \& Computational Sciences \\
[-0.8ex]\small Urcuqu\'i, Ecuador.  
\\[-0.8ex]\small\tt
mmendez@yachaytech.edu.ec\\ [-0.8ex]\small\tt mmendezenator@gmail.com}
\title{Shift-plethysm, Hydra continued fractions, and $m$-distinct partitions}

\begin{document}
\maketitle
\begin{abstract}We introduce the hydra continued fractions, as a generalization of the Rogers-Ramanujan continued fractions, and give  a combinatorial interpretation in terms of shift-plethystic trees. We then show it is possible to express them as a quotient of $m$-distinct partition generating functions, and in its dual form as a quotient of the generating functions of compositions with contiguous rises upper bounded by $m-1$. We obtain new generating functions for compositions according to their local minima, for partitions with a prescribed set of rises, and for compositions with  prescribed sets of contiguous differences.
\end{abstract}

\\
\noindent\keywords{Continued fractions, Integer partitions, Integer compositions, Rogers-Ramanujan identities,   non-commutative series.}\\

\section{Introduction} 
 Let us consider two formal power series $F$ and $G$ depending in an infinite number of commuting variables $\x=(x_1,x_2,x_3,\dots)$, $G$ having zero constant term. The Polya plethysm \cite{Polya1937, joyal1981theorie}, is defined as follows
 	\begin{equation*}
 	 (F\circ_P G)(\x)=F(G(\x),\free_2G(\x),\free_3G(\x),\dots)
 	 \end{equation*}
 	 \noindent where $\free_n$ is the Frobenius operator,  $\free_nG(x_1,x_2,x_3\dots):=G(x_n,x_{2n},x_{3n},\dots)$.
 	 If we consider instead series depending on  variables indexed by natural numbers, $\x=(x_0,x_1,x_2,\dots)$ and the shift operators $\sh^n G(x_0,x_1,x_2,\dots)=G(x_n,x_{n+1},x_{n+2},\dots)$, $n=1,2,3\dots$, instead of Frobenius, we get the operation of shift-plethysm,
 	 \begin{equation*}
 	 (F\spl G)(\x)=F(G(\x),\sh G(\x),\sh^2 G(\x),\dots).
 	 \end{equation*}
 	 The monoidal structure subjacent in each of these two plethysms is apparent. That is, the positive integers with the product in the former, and the additive structure of the natural numbers in the second. A general form of plethysm, with variables in cancellative monoids, was introduced in \cite{Mendezava}. For the shift-plethysm, this general interpretation has proven surprisingly helpful in the enumeration of combinatorial objects, where the additive structure of the natural numbers is the key element. Informally, {\em a numerical structure} will be  one described by using a subset of the natural numbers and its additive properties. Prototypical examples of families of numerical structures are integer compositions and partitions. The umbral map, $x_n\mapsto x^n$ in the Polya's cycle index polynomial of a group, leads to the enumeration of unlabeled combinatorial objects over which the group acts (the number of orbits under the action of the group). In this case, the plethysm corresponds to the cycle index of the wreath product of two groups. In Joyal's theory of species, Polya's plethysm is associated with the operation of substitution. A family of labeled combinatorial structures (species) is assigned a cycle index series in an infinite number of variables. The series (in one variable) enumerating the unlabeled structures  is obtained  by the umbral map applied to the cycle index series. Polya's plethysm is closely related to the operation of substitution of species. Informally, the elements of the substitution of one species  into another, are the combinatorial objects of the former family placed  \emph{inside} the combinatorial objects of the second family of structures. However, the enumeration of the unlabeled structures of the substitution can not be obtained by the  simple substitution of one of the generating series of unlabeled structures into the other. It is necessary to go back to the cycle index series, compute their plethysm and then apply the umbral map \cite{Bergeron1998}.
 	 
 	 Something analogous occurs with shift-plethysm. It enumerates shifted families of numerical structures inside another given family of numerical structures and the umbral map $x_n\mapsto zq^n $, $n=0,1,2,\dots$, sends series in an infinite number of  variables to $q$-series. Considering two particular kinds of series in infinite variables, by shift-plethysm and umbral mapping we recover two classical operations of substitution of $q$-series (see \cite{MendezJuly2020} Section 6.1). Although, part of the original combinatorial meaning implicit in shift-plethysm is obviously lost in this passage from infinite variables to $q$-series. 
 	 
 	 Shift-plethysm is straightforwardly extended to non-commutative series in the alphabet, $$\X=\{X_0, X_1,X_2,\dots\}.$$ But in  fact, our approach in this article goes in the opposite direction. We define first shift-plethysm for noncommutative series, and then project to infinite commutative variables and $q$-series by abeleanization and  umbral map respectively. 
 	  
 	 Noncommutative continued fractions began to be studied as early as 1913 by Wedderburn (see \cite{Wedderburn1913}). In
 	    \cite{Flajolet1980}, Flajolet gave an interesting and nice combinatorial interpretation of the the general Stieltjes-Jacobi continued fractions in non-commuting variables in terms of labeled paths, and solved various enumeration problems. In \cite{Pak1995}, a noncommutative version of the Rogers-Ramanujan continued fraction was given and a general noncommutative Lagrange inversion formula is discussed in connection with the theory of quasideterminants (see  \cite{gel1992theory}, and \cite{Gelfand2005}).
 	    In \cite{Rogers1893} Rogers presented what is now known as the Rogers-Ramanujan continued fraction $\mathcal{R}$, expressed here as a $q$-series, $\mathcal{R}(z)=\mathcal{R}(z,q)$,
 	    $$\mathcal{R}(z)=\frac{z}{1+z\cfrac{q}{1+z\cfrac{q^2}{\ddots}}}$$
 	    and proved that  \begin{equation}\label{eq.RR1}\mathcal{R}(z)=z\frac{\sum_{n=0}^{\infty}\frac{q^{n(n+1)}z^n}{(1-q)(1-q^2)\dots (1-q^n)}}{\sum_{n=0}^{\infty}\frac{ q^{n^2}z^n}{(1-q)(1-q^2)\dots (1-q^n)}}.\end{equation}
 	    We drop the $q^{\frac{1}{5}}$ factor in original Rogers-Ramanujan continued fraction because our main concern here is about its combinatorial meaning. As it was first pointed out by MacMahon and Schur \cite{MacMahonbook, Schur1917}, the numerator and denominator of Eq. (\ref{eq.RR1}) are $q$-generating functions of partitions (in increasing order) with rises lower bounded by $2$ ($2$-distinct partitions). The generating function of the  numerator counting $2$-distinct partitions with it first part at least two. These  generating functions are the main characters in the Rogers-Ramanujan identities (see \cite{sills2017invitation}, and references therein).
 	    In \cite{MendezJuly2020} we consider the following non-commutative generalization of the Rogers-Ramanujan continued fraction, defined by the equation
 	    \begin{equation*}
 	    \mathcal{R}=X_0\frac{1}{1+\sh\mathcal{R}}=X_0\left(\frac{1}{1+X_1}\spl\mathcal{R}\right).
 	    \end{equation*}
 	    The umbral map gives   $\mathcal{R}(z)$ again. 
 	    
 	     As a natural generalization, we define an $m$-headed {\em hydra continued fraction} (hydra fraction, for short)  by the implicit shift-plethystic equation
 	    \begin{equation}\label{eq.hydrafraction}
 	    \mathcal{R}_m=X_0\left(\frac{1}{(1+X_m)(1+X_{m-1})\dots(1+X_1)}\spl \mathcal{R}_m\right).
 	    \end{equation}
 	    The main result in the present article is the introduction of the hydra continued fractions as a generalization of  Rogers-Ramanujan continued fraction, their combinatorial meaning, and their enumerative applications.
 	    
 	    Clearly we have $\mathcal{R}_1=\mathcal{R}$. We prove that the hydra fraction $\mathcal{R}_{m-1}$ can be expressed as a quotient of two generating functions of partitions having rises lower bounded by $m$  ($m$-distinct partitions), see Theorem \ref{theo.quotientpart}. The generating function in the numerator being that of $m$-distinct partitions with least part greater than or equal $m$, thus generalizing Eq. (\ref{eq.RR1}). For example, for $m=2$, the $2$-headed hydra fraction can be expressed as the quotient of the respective generating function of the $3$-distinct partitions, as follows,
 	    \begin{equation*}\mathcal{R}_2(z)=\frac{z}{\left(1+\frac{zq^2}{\left(1+\frac{zq^4}{\iddots\ddots}\right)\left(1+\frac{zq^3}{\iddots\ddots}\right)}\right)\left(1+\frac{zq}{\left(1+\frac{zq^3}{\iddots\ddots}\right)\left(1+\frac{zq^2}{\iddots\ddots}\right)}\right)}=z\frac{\sum_{n=0}^{\infty}\frac{q^{\frac{3n(n+1)}{2}}z^n}{(q;q)_{n}}}{\sum_{n =0}^{\infty}\frac{q^{\frac{n(3n-1)}{2}}z^n}{(q;q)_n}}.\end{equation*}
 	
 	\noindent Here we use the Pochhammer symbol $(a,q)_n=(1-a)(1-aq)\dots(1-aq^{n-1})$.
 	 If we change in Eq. (\ref{eq.hydrafraction}) the sign of each variable $X_n$, $n\geq 1$, we obtain the generating function of the shift-plethystic trees enriched with partitions with parts upper bounded by $m$, $\msA_{\Pib_{m}}=-\mathcal{R}_{m}(-X)$. We prove the dual result for the hydra fraction $\msA_{\Pib_{m-1}}$. It is the quotient of two generating functions of compositions with contiguous differences upper bounded by $m-1$ (see Corollary \ref{cor.quotientcomp}). We consider also the infinitely headed hydra fraction, enumerating shift-plethystic trees enriched with partitions of any size, leading to the enumeration of compositions according with their local minima (see Theorem \ref{theo.mcompositionsfactor}, and Theorem \ref{theo.complocalminima}). We also consider branchless shift-plethystic trees. This construction allows us, by using shift-plethystic inversion, to find a general formula for partitions with rises in a subset of $\N$ (see Theorem \ref{teo.brnchlespart}).  A dual formula, for the enumeration of compositions with contiguous differences in the complementary set in $\Z$ is given in Corollary \ref{cor.compositionsSc}.

 	As a guide to the reader, in Appendix \ref{table.notation} we give a list with the notation for the most relevant series in the article.
\section{Non-commutative series}
Let $\mathbb{A}$ be be an alphabet (a totally ordered set) with at most a countable number of elements (letters). Let $\A^*$ be the free monoid generated by $\A$. It consists of words or finite strings of letters in $\A$, $\om=\om_1 \om_2\dots \om_n$, including de empty string represented as $1$. We denote by $\ell(\om)$ the length of $\om$.
Let $\K$ be a field of characteristic zero. A \emph{noncommutative} formal power series in $\A$ over $\K$ is a function $R:\A^*\rightarrow \K$. We denote $R(\om)$ by $\la R,\om\ran$ and represent $R$ as a formal series

$$R=\sum_{\om\in\mathbb{A}^*}\langle R,\om\rangle\,\om, \;\langle R,\om\rangle\in\mathbb{K},$$

The sum and product of two formal power series $R$ and $S$ are respectively given by 
\begin{eqnarray*}
	R+S&=&\sum_{\om\in\A^*}(\la R,\om\ran+\la S,\om\ran) \om\\R.S&=&\sum_{\om\in \mathbb{A}^*}(\sum_{\om_1\om_2=\om}\langle R,\om_1\rangle \langle S,\om_2\rangle) \omega. \end{eqnarray*}
The algebra of noncommutative formal power series is denoted by $\K\la\la\A\ran\ran$.
There is a notion of convergence on $\K\la\la\A\ran\ran$. We say that $R_1, R_2, R_3,\dots$ converges to $R$ if for all $\om\in \A^*$, $\la R_n,\om\ran= \la R,\om\ran$ for $n$ big enough. A {\em language} (on $\A$) is a subset of $\A^*$. We identify a language $L$ with its generating function, the formal power series
$$L=\sum_{\om\in L} \om.$$
The support of a series $R$  is the language of words where $R$ is different from zero, $$\mathrm{supp}(R)=\{\omega|\la R,\omega\ran\neq 0\}$$
 If $\la R,1\ran=\alpha\neq 0$, then $R$ has an inverse given by (see for example \cite{Stanley1999})
$$R^{-1}=\frac{1}{\alpha}\sum_{n=0}^{\infty}\left(1-\frac{R}{\alpha}\right)^n.$$
Let $B$ be a series having constant term equal to zero, $\la B,1\ran =0$. We denote by $\frac{1}{1-B}$, the inverse of the series $1-B$,
$$\frac{1}{1-B}:=(1-B)^{-1}=\sum_{k=0}^\infty B^k.$$
\section{The algebra $\K\langle\langle \X\rangle\rangle$}
Denote by $\X$ the alphabet $\{X_0, X_1, X_3,\dots\}$. The words in the algebra $\K\langle\langle \X\rangle\rangle$ are indexed by weak compositions. Let $\kap=(k_1,k_2,\dots,k_\ell)$ be an element of $\N^\ell$ (a weak composition). We denote by $X_{\kap}$ the word $X_{k_1}X_{k_2}\dots X_{k_\ell}$, the empty word denoted by $1$. We denote by $|\kap|$ the sum of the parts of $\kappa$, $$|\kap|=k_1+k_2+\dots+k_\ell$$ and by $\ell(\kap)=\ell$ its length.
A formal power series in $\K\langle\langle \X\rangle\rangle$ then has the form
$$ R=\sum_{\kap}\langle R,X_{\kap}\rangle X_{\kap}$$
Let $S$ be a subset of $\N$. We denote by $\Al_S$ the language formed by the single letters $X_k$, $k\in S$,
$\Al_S=\sum_{k\in S}X_k.$

As special cases we denote
 $\Al_m=\sum_{k=m}^{\infty}X_k,\;\mbox{ and }\; \Al_{m}^n=\sum_{k=m}^{n}X_k.$
We shall call $\kap$ a (strong) composition if $k_i\neq 0$, for every $i$.
Observe that the language of strong compositions $\C$ is given by the series
$\C=\Al_1^*=\frac{1}{1-\Al_1}.$
 In what follows the word {\em composition} will mean by defect {\em  strong composition}.
 
We denote by $\Pib_m$ the series of partitions in decreasing order, allowing repetitions and with longest part less than or equal to $m$, 
$$\Pib_{m}^1=\Pib_m:=\prod_{k=m}^1\frac{1}{1-X_k}.$$
The limit $\Pib_{\infty}:=\lim_{n\rightarrow\infty}\Pib_m$ is the series of partitions with parts of any size (in decreasing order),
$$\Pib_{\infty}=\Pib_{\infty}^1:=\prod_{k=\infty}^1\frac{1}{1-X_k}=\lim_{m\rightarrow\infty}\prod_{k=m}^{1}\frac{1}{1-X_k}$$
Series of partitions without repetitions (in increasing order) will be denoted by the symbol $\Pi$, 
$$\Pi^m:=\prod_{k=1}^m(1+X_k),\; \Pi^{\infty}:=\prod_{k=1}^{\infty}(1+X_k)=\lim_{n\rightarrow\infty}\Pi^m.$$
The {\em abeleanization} is the algebra map  $\ab:\K\langle\langle \X\rangle\rangle\rightarrow\K[[x_0,x_1,x_2,\dots]],$ defined by  $\ab(X_k)=x_k$.
The {\em umbral map} $\um\ab(X_k)=\um(x_k)=zq^k$
sends a series $S$ in infinite variables to the $q$-series 
\begin{equation*}
S(z)=S(z,zq,zq^2,zq^3,\dots),
\end{equation*}
which by abuse of language we shall denote with the same symbol $S$.
We have
\begin{equation*}
S(z)=\sum_{k=0}^\infty(\sum_{\ell(\kap)=k}\la S,X_{\kap}\ran q^{|\kap|})z^k.
\end{equation*}  
Then $\um$ is an algebra  map from $\K[[x_0, x_1, x_2,\dots]]$ to $\K[[q]][[z]]=\K[[q,z]].$
 \subsection {Linked languages}
We consider now a special kind of languages obtained from a given set of `links' $B\subseteq W\times W$, where $W$ is some fixed subset of $\N$. Define
$$L_B=\{X_{\kap}|(k_i,k_{i+1})\in B, \mbox{ for every }i=1,2,\dots,\ell(\kap)-1\},$$
and the language $L$ associated to $B$ by
\begin{equation}\label{eq.linked}L=1+\Al_W+L_B.
\end{equation}
We shall call an $L$ of this form a
{\em {\em  linked language}}. 
Define the K-dual $L^!$ to be the language associated with the complement set of links
$$L^!=1+\Al_W+L_{B^c}$$
For linked languages we define a second formal power series, 
$$L^{g}=\sum_{\kap\in L}(-1)^{\ell(\kap)}X_{\kap}.$$ 
We call it the {\em {\em  graded}} generating function of $L$. Eq. (\ref{eq.kdual}) gives us an inversion formula for linked languages.

\begin{proposition}\label{prop.kdual1} The series $L^!$ is given by the inverse of the graded generating function of $L$,
	\begin{equation}\label{eq.kdual}
	L^!=(L^g)^{-1}.
	\end{equation} 
\end{proposition}
 Formula (\ref{eq.kdual}) is a non-commutative version of  Theorem 4.1. in Gessel PhD thesis, \cite{Gesselthesis}, where the terminology of linked sets was used for the first time. See \cite{MendezJuly2020} for a proof of Proposition \ref{prop.kdual1}. It  is a particular instance of the inversion formulas relating generating functions of two dual Koszul algebras. For the interested reader, Koszul algebras were introduced by Priddy  in \cite{priddy1970koszul}. A detailed study of Koszul algebras and inversion formulas could be found in  \cite{polishchuk2005quadratic}.

 \begin{example}\label{ex.dualpm}{\em  Compositions and $m$-distinct partitions.} We denote by $\Pa_m$ the language of partitions (in increasing order) of the form $\boldsymbol{\lam}=\lam_1\leq\lam_2\leq\dots$, such that $\lam_{i+1}-\lam_{i}\geq m$ for $m$ a non negative integer. Those kind of partitions are called $m$-distinct in \cite{Andrews2004}.  
\begin{equation*}
\Pa_{m}=1+\Al_1+\sum_{\lam_{i+1}-\lam_i\geq m} X_{\boldsymbol{\lam}}.
\end{equation*} 
We denote by $\C^{(m)}$ the language of compositions with upper bounded contiguous differences
\begin{equation*}
\C^{(m)}=1+\Al_1+\sum_{k_{i+1}-k_i\leq m} X_{\kap}.
\end{equation*} 
Clearly we have the duality $\Pa_m^!=\C^{(m-1)}$ and we get,
\begin{equation}\label{eq.partcomdual}
\C^{(m-1)}=(\Pa_m(-X))^{-1}.
\end{equation}
We have that (see \cite{lehmer1946two}, Theorem 1)
\begin{equation}\label{eq.mdistinct}
\Pa_m(z)=\sum_{k=0}^{\infty}\cfrac{q^{m\binom{k}{2}+k}z^k}{(q;q)_k}
\end{equation}
We get
\begin{equation}\label{eq.mm1comp}
\C^{(m-1)}(z)=\left(\sum_{k=0}^{\infty}\cfrac{(-1)^k q^{m\binom{k}{2}+k}z^k}{(q;q)_k}\right)^{-1}
\end{equation}
This identity was proved by Zeilbeger for $m=2$ (see the sequence A003116 in OEIS), in the context of a formula of Lehmer for the determinant of a tridiagonal matrix \cite{Ekhad2019}. 
\end{example}
\begin{example}{\em Carlitz compositions.} Consider the set of links $B=\{(i,i)|i\in\Al_1\}$. The corresponding linked language is that of words using one repeated letter in the alphabet $\Al_1$,  
	\begin{equation*}\mathrm{O}=1+\sum_{i=1}^{\infty}\sum_{k=1}^{\infty}X_{i}^k=1+\sum_{i=1}^{\infty}\frac{X_i}{1-X_i}.
	\end{equation*}
	Its $K$-dual,  associated with the set $B^{c}=\{(i,j)|j\neq i+1\}$, is the language of Carlitz compositions, with words without contiguous  repeated letters (see \cite{ Carlitz1976, Heubach2009}),
	\begin{equation*}\mathrm{C}=(\mathrm{O}(-X))^{-1}=\frac{1}{1+\sum_{i=1}^{\infty}\frac{-X_i}{1-(-X_i)}}=\frac{1}{1-\sum_{i=1}^{\infty}\frac{X_i}{1+X_i}}.
	\end{equation*}
\end{example}

\section{Shift-plethysm}
Let $\kap$ be a composition and $n$ a non negative  integer $n$, and assume that every component of $\kap$ is greater than or equal to $n$. We denote this fact by the inequality $\kap\geq n$ and define $\kap-n$ to be the (in general weak) composition $(k_1-n,k_2-n,\dots,k_{\ell}-n)$.
\begin{definition}
	Define $$\sigma:\K\la\la\X\ran\ran\ra\K\la\la\X\ran\ran$$ by extending the shift   $$\sigma X_i=X_{i+1}, \;i=0,1,2,.\dots$$ as a continuous algebra map. Equivalently, define for a series $R$,
	$$\la \sh R|X_{\kap}\ran:=\begin{cases}\la R|X_{\kap-1}\ran& \mbox{ if $\kap\geq1$}\\0&\mbox{ otherwise.}\end{cases}$$ 
	In general, for a non negative integer $n$,
	$$\la \sh^n R|X_{\kap}\ran=\begin{cases}\la R|X_{\kap-n}\ran& \mbox{ if $\kap\geq  n$}\\0&\mbox{ otherwise.}\end{cases}$$
\end{definition}

\begin{definition} Let $S$ be a series in $\K\laa \X\raa$. 
	We define the  {\em  {\em shift-plethystic substitution}} of $S$ in a word $X_{\kap}=X_{k_1}X_{k_2}X_{k_3}\dots X_{k_\ell}$, as the substitution of the shift $\sigma^{k_i}$ on each of the letters of $X_{\kap}$,
	\begin{equation*}
	X_{\kap}\spl S=(\sigma^{k_1}S)(\sigma^{k_2}S)\dots(\sigma^{k_\ell}S).
	\end{equation*}
	In particular we have that $X_n\spl S=\sh^n S$.
	For a formal power series $R$, and $S$ with  zero constant term, $\la S,1\ran=0$, define the shift-plethysm $R\spl S$ by 
	\begin{equation}\label{eq.shiftplethys}
	R\spl S=\sum_{\kap\in\N^*}\la R,X_{\kap}\rangle X_{\kap}\spl S=\sum_{\kap\in\N^*}\la R,X_{\kap}\rangle  (\sigma^{k_1}S)(\sigma^{k_2}S)\dots(\sigma^{k_\ell}S).
	\end{equation}
\end{definition}
Shift-plethysm is well defined, the series in the right hand side of Eq. (\ref{eq.shiftplethys}) is convergent (see \cite{MendezJuly2020}).
It is associative and non-commutative. The one letter series
 $X_n$ commutes with any other, 
\begin{equation*}
X_n\spl S=S\spl X_n.
\end{equation*} 
By associativity we also have
\begin{equation*}
X_n\spl R\spl S=R\spl X_n\spl S=R\spl S\spl X_n.
\end{equation*} 
\subsection{Shift-plethysm and $q$-series}
The abelianization of the shift-plethysm is obtained by the  substitution of $S(x_n,x_{n+1},x_{n+2},\dots)$,
in the variable $x_n$ of the series $R(x_0,x_{1},x_{2},\dots)$
\begin{equation*}
(R\spl S)(\x)=R(S(x_0,x_1,x_2,\dots),S(x_1,x_2,x_3,\dots),S(x_2,x_3,x_4,\dots),\dots).
\end{equation*} 
Since $\sh S(z)=S(zq,zq,zq^2,\dots)=S(zq)$ we have 
\begin{equation*}
(R\spl S)(z)=R(S(z),S(zq),S(zq^2),\dots).
\end{equation*} 
In \cite{MendezJuly2020} we explain how this formula generalizes two classical notions of $q$-substitution
 (see for example that in \cite{garsia1981qLag}). 
\begin{example}{\em Carlitz Compositions.} 
	\begin{figure}[hbt!]
		\begin{center}\includegraphics[width=95mm]{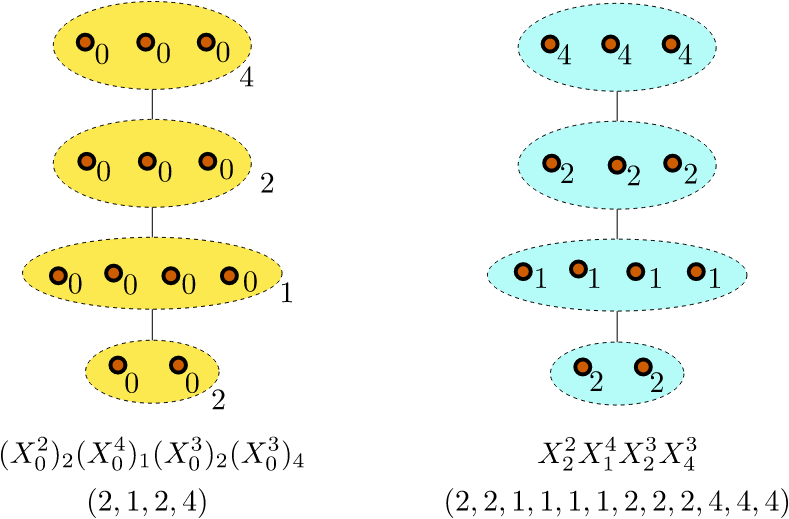}
		\end{center}\caption{A pictorial representation of the identity  $\C=\mathrm{C}\spl \frac{X_0}{1-X_0}$.}\label{fig.mac}
	\end{figure}
	 Given a composition $\boldsymbol{\kap}$, we say that $i$ is a distinction of $ \boldsymbol{\kap}$ if $k_i\neq k_{i+1}$. We can factor $ \boldsymbol{\kap}$ by placing bars after every distinction, and we obtain a composition of the form
	$$ \boldsymbol{\kap}=\mu_1\mu_1\dots\mu_1|\mu_2\mu_2\dots\mu_2|\dots|\mu_k\mu_k\dots\mu_k.$$
	Hence $\boldsymbol{\mu}=(\mu_1,\mu_2,\dots,\mu_k)$ is a Carlitz composition.
	Denote by $\tau_i$ the number of repetitions of $\mu_i$, $i=1,2,\dots,k$. 
	We have that
	$$X_{\boldsymbol{\kap}}=(X_{\mu_1}\spl X_0^{\tau_1})(X_{\mu_2}\spl X_0^{\tau_2})\dots (X_{\mu_k}\spl X_0^{\tau_k}),$$
	which is a word in the language $$X_{\boldsymbol{\mu}}\spl \sum_{k=0}^{\infty}X_0^k=X_{\boldsymbol{\mu}}\spl\frac{X_0}{1-X_0}$$
	Since every composition in $\C$ has a similar factorization for some Carlitz composition $\boldsymbol{\mu}$, we have
	$$\C=\sum_{\boldsymbol{\mu}\in\mathrm{C}}X_{\boldsymbol{\mu}}\spl\frac{X_0}{1-X_0}=\mathrm{C}\spl\frac{X_0}{1-X_0}$$ 
	The shift-plethystic inverse of $\frac{X_0}{1-X_0}$ is $\frac{X_0}{1+X_0}$. Hence,
	 
	$$\mathrm{C}=\C\spl \frac{X_0}{1+X_0}=\frac{1}{1-\Al_1\spl  \frac{X_0}{1+X_0}}=\frac{1}{1-\sum_{i=1}^\infty\frac{X_i}{1+X_i}}.$$
	
	See Fig. \ref{fig.mac},  where the shift-plethystic substitution of words in  the language $\frac{X_0}{1-X_0}$ into the word corresponding to the Carlitz composition  $(2,1,2,4)$ is pictorially represented.\end{example}
\begin{proposition}We have the following identity
\begin{equation}\label{eq.shplethymtoone}	\Pa_m\spl X_0\Pi^{m-1}=\Pi^{\infty}=\Pa_1
\end{equation}
\noindent for arbitrary integer $m\geq 1.$
\end{proposition}

\begin{proof}Observe that the language $X_0\Pi^{m-1}=X_0\prod_{k=1}^{m-1}(1+X_k)$ has as elements the words of the form $X_0X_{\boldsymbol{\lam}}$, where $\lam_{i+1}-\lam_{i}>1$, and $\lam_\ell\leq m-1$, $\ell\leq m-1$ being the length of $\boldsymbol{\lam}$. Recall that the series $\Pa_1=\Pi^{\infty}$ is the language of partitions with distinct parts (in increasing form). Let $X_{\boldsymbol{\tau}}$ be a word in $\Pi^{\infty}$. Define recursively the partition $\boldsymbol{\mu}$ as follows. Define $\mu_1$ to be the first element of $\boldsymbol{\tau}$. After having defined $\mu_i$ for $i<r$, define $\mu_r=\tau_j$ where $j=\mathrm{min}\{s|\tau_s-\mu_{r-1}\geq m\}$. Then, placing bars before each $\mu_i,\, i>2$, $\boldsymbol{\tau}$ can be uniquely factored as follows
	$$\boldsymbol{\tau}=\mu_1\boldsymbol{\lam}_1|\mu_2\boldsymbol{\lam}_2|\dots|\mu_k\boldsymbol{\lam}_k.$$  
\noindent where by construction $\mu_{i+1}-\mu_i\geq m$, and each $\boldsymbol{\lam}_i-\mu_i$  is in the language $\Pi^{m-1}$. Hence, every word $X_{\boldsymbol{\tau}}$ can be uniquely factored as follows $$X_{\boldsymbol{\tau}}=(X_{\mu_1}\spl X_0 X_{\boldsymbol{\lam}_1-\mu_1})(X_{\mu_2}\spl X_0 X_{\boldsymbol{\lam}_2-\mu_2})\dots (X_{\mu_1}\spl X_0 X_{\boldsymbol{\lam}_k-\mu_k}).$$
Then we have 
$$\Pa_1=\Pi^{\infty}=\sum_{\boldsymbol{\mu}\in\Pa_m}X_{\boldsymbol{\mu}}\spl X_0\Pi^{m-1}=\Pa_m\spl X_0\Pi^{m-1}.$$
\end{proof}

\subsection{Implicit Equations}
\begin{definition}{\em Implicit shift-plethystic equation}. Let $\mathbb{Y}=\{Y_0,Y_1,Y_2,\dots\}$ be an alphabet disjoint with $\X$. Consider the implicit equation
\begin{equation}\label{eq.implicit}
Y_0=F(X;Y) 
\end{equation}
\noindent where $F(X,Y)$ is a noncommutative formal power series in the alphabet $\X\cup\Y$, satisfying
\begin{enumerate}
	\item $F$ does not have constant term, $\langle F,1\rangle=0.$
	\item The coefficient of $F$ in $Y_0$ is equal to zero, $\langle F,Y_0\rangle =0$.
\end{enumerate}
An equation as above will be called a {shift-plethystic implicit equation}.
\end{definition}
\begin{definition}
We say that a noncommutative series $G=G(X)$, having zero constant term, is a solution of the shift-plethystic equation  (\ref{eq.implicit}) if after the substitution of $Y_k$ by $\sh^kG(X)$  we get the formal power series identity
\begin{equation}\label{eq.implicit2}G(X)=F(X;Y)_{Y_k=\sh^k G(X)}=F(X; G(X),\sh G(X),\sh^2 G(X),\dots).\end{equation}
\end{definition}
By  applying  the shifting $\sh^r$ to both sides of Eq. (\ref{eq.implicit2}) we can see that the implicit equation (\ref{eq.implicit}) is indeed equivalent to the infinite system
\begin{equation*}
Y_r=F(X_r,X_{r+1},X_{r+2},\dots;Y_{r},Y_{r+1},Y_{r+2},\dots),\;r=0,1,2,\dots.
\end{equation*} 
By simplicity, we shall denote by $F(X;G(X))$ the shift-plethystic substitution in the right hand side of Eq. (\ref{eq.implicit2}).
\begin{proposition}\label{prop.implicit}\normalfont{Every shift-plethystic equation has a unique solution $G(X)$.}
\end{proposition}
\begin{proof}
	 {\em  Sketch of the proof}. The sketch of the proof is standard, similar to the proof of the existence and unicity of the solution for the implicit equations defining an algebraic language (see for example \cite{Stanley1999, Eilenberg1974}). See also the proof of the implicit function theorem for species \cite{joyal1981theorie}, and its combinatorial interpretation in \cite{Bergeron1998}. However, several technical details of the proof and its combinatorial interpretation are inherent to shift-plethysm. They are given in the Appendix \ref{sec.appendix}. 
	We define

	\begin{equation*}
	\begin{cases}G^{(0)}(X)&=0\\
	G^{(n+1)}(X)&=F(X;Y)_{Y_k=\sh^k G^{(n)}},\mbox{ for $n\geq 0$.}\end{cases}
	\end{equation*}
	We have that $G^{(n)}(X)$ converges, its limit $$G(X):=\lim_{n\rightarrow \infty}G^{(n)}(X)=\lim_{n\rightarrow\infty}F(X;G^{(n-1)})=F(X;G(X)),$$ is a solution of Eq. (\ref{eq.implicit}),  and this solution is unique. 
\end{proof}

\subsection{Shift-plethystic inverse and $\Lau$-series}
From Proposition \ref{prop.implicit} we get a necessary and sufficient condition for a series to have a shift-plethystic inverse. 
\begin{proposition}Let $R$ be a power series without constant term, $\la R,1\ran=0$. Then, $R$ has a shift-plethystic inverse in $\K\la\la \X\ran\ran$ if and only if  $\la R,X_0\ran =\alpha\neq 0$.
	\end{proposition}
\begin{proof} Let $R$ be a series without constant term. Since $\la R\spl R^{\la -1\ran},X_0\ran=\la R,X_0\ran\la R^{\la -1\ran},X_0\ran$, it is easy to check that $\la R,X_0\ran\neq 0$ is a necessary condition for $R$ to have a shift-plethystic inverse. Assume now that $\la R,X_0\ran=\alpha \neq 0$ and define $\free_R$ by means of the implicit equation
	\begin{equation}\label{eq.implicitinverse}\free_R=\frac{1}{\alpha}(X_0-R^+\spl\free_R).\end{equation}
\noindent where $R^+=R-\alpha X_0$. This implicit equation is as in Proposition \ref{prop.implicit}, with $$F(X,Y)=\frac{1}{\alpha}(X_0-R^+(Y)).$$ Which clearly satisfy the condition  $\la F,Y_0\ran=0$. From Eq. (\ref{eq.implicitinverse}) we obtain
\begin{equation*}
\alpha\free_{R}+R^{+}\spl\free_{R}=(\alpha X_0+R^+)\spl\free_R=R\spl \free_R=X_0,
\end{equation*} 
\noindent which means that $\free_R=R^{\la -1\ran}$.\end{proof}

We have that $\{\sh^j|j\in \N\}$ is a monoid of operators acting on $\Sho$. In order to extend it to the group $\{\sh^j|j\in \Z\}$ and consider negative shifts, we have to extend our alphabet $\X$ to
 $$\Xpm=\{X_0, X_{\pm 1}, X_{\pm2},\dots\}.$$  
In $\K\la\la \Xpm\ran\ran$ we can define the inverses of the shift operator as the algebra map that continuously extends
$$\sh^{-1}X_m=X_{m-1},\; m\in \Z.$$ 
Shift-plethysm is not well defined in $\K\la\la \Xpm\ran\ran$. For example, the computation of coefficients shift-plethysm of the series $\sum_{n\in \Z}X_n$ with itself involves infinite sums of positive coefficients.
However, we can construct an extension of the algebra $\Sho$, where the group of shifts operators acts, and the shift-plethysm is still well defined. 

\begin{definition}
Let $n$ be an integer. We define $\La_n$ as the vector space of shifted formal power series, $\sh^n \Sho$. We denote by  $\Lau$  the sum as vector spaces of all $\La_n$, $n\leq 0$.    

\begin{equation*}\Lau=\sum_{n=-\infty}^{0} \La_n=\sum_{n=-\infty}^{0} \sh^{n}\Sho\end{equation*} 
It is clear that $\Lau$ is an sub-algebra of $\K\la\la \Xpm\ran\ran$.
 \end{definition}
To prove that shift-plethysm is a well defined operation in $\Lau$, we need to define the {\em order} of a formal power series.
\begin{definition}Let $\kap$ be a word in the alphabet $\X_{\pm}$. Define the order of $X_{\kap}$ to be the minimun of the components of $\kap$. 
	\end{definition}

The set of orders of the non-empty words in the support of a series $R$ is bounded below if and only if and only if $R$ is in $\Lau$. We then define for  $R\in\Lau$ a non-constant series
\begin{equation*}
\ord(R)=\mathrm{min}\{\ord(\kap)|\la R,X_\kap\ran\neq 0\}.
\end{equation*} 
The shift-plethysm $R\spl S$, $\la S,1\ran =0$, is naturally extended from Eq. (\ref{eq.shiftplethys}) to series in  $\Lau$, by including words having possible negative components,
\begin{equation*}
R\spl S=\sum_{\kap\in\Z^*}\la R,X_{\kap}\ran X_{\kap}\spl S.
\end{equation*}

\begin{proposition}\label{prop.order}
\normalfont{For $R$ and $S$ series in $\Lau$, we have that
	 if $\la S,1\ran =0$, then $R\spl S$ is well defined.}
\end{proposition}
\begin{proof}
	We have to prove that for arbitrary $\boldsymbol{\tau}$, the sum
	$$\sum_{\kap\in\Z^*}\la R,X_{\kap}\ran \la X_{\kap}\spl S,X_{\boldsymbol{\tau}}\ran$$
	has only a finite number of nonzero terms. Since
	$$\la X_{\kap}\spl S,X_{\boldsymbol{\tau}}\ran=\sum_{\boldsymbol{\tau}^{(1)}\boldsymbol{\tau}^{(2)}\dots\boldsymbol{\tau}^{(\ell(\kap))}=\boldsymbol{\tau}}\prod_{i=1}^{\ell(\kap)}\la S,X_{\boldsymbol{\tau}^{(i)}-k_i}\ran$$
 If we assume that this expression is different from zero, we should have that
$\ord(R)\leq k_i\leq \mathrm{max}\{\tau_i|i=1,2,\dots,\ell(\boldsymbol{\tau})\}-\ord(S)\,\mbox{ and }\ell(\kap)\leq\ell(\boldsymbol{\tau}),$
which involves only a finite number of $\kap's$.	
\end{proof}
\begin{remark}\normalfont{The set of series in $\Lau$ of the form $\sum_{k=n}^{\infty}c_kX_k$, $n\in\Z$, is closed with with respect to the operation of  shift-plethysm. It is isomorphic the the ordinary Laurent formal power series (with the product) by the umbral map
		\begin{equation*}
		\sum_{k=n}^{\infty}c_kX_k\mapsto \sum_{k=n}^{\infty}c_kq^k
		\end{equation*}  }
\end{remark}
The umbral map, $X_k\mapsto zq^k$, sends a series $R$ into  a $q$-series $R(z)$ with coefficients in the field of Laurent formal power series in the indeterminate $q$. If the order of $R$ is a negative integer $-n\in\Z$, we have that the umbral map sends $R\spl S$ to the $q$-series
\begin{equation*}
(R\spl S)(z)=R(S(zq^{-n}),S(zq^{-n+1}),\dots,S(z),S(zq),S(zq^2,\dots)).
\end{equation*} 
\begin{proposition}\normalfont{Let $R$ be a series in $\Lau$ without constant term. Then, the following conditions are equivalent
	\begin{enumerate}
\item $R$ is invertible with respect to shift-plethysm.
\item $\la R,X_n\ran\neq 0$, $n\in\Z$ being the order of $R$.
\item $\sh^{-n}R$ is invertible in $\Sho$.
	\end{enumerate}}
\end{proposition}\begin{proof}
Easy and left to the reader. 
\end{proof}
 \begin{example}\label{ex.inversesigma}
 	The formal power series $\Al_n$, $n\in\Z$, is invertible in $\Lau$
 	$$(\Al_n)^{\langle -1\rangle}=X_{-n}-X_0.$$
 	The series of non-empty compositions $$\C_+=\frac{\Al_1}{1-\Al_1}=\frac{X_0}{1-X_0}\spl\Al_1$$
 	has as inverse
 	$$(\C_+)^{\la -1\ran}=(\Al_1)^{\la -1\ran}\spl \frac{X_0}{1+X_0}=(X_{-1}-X_0)\spl\frac{X_0}{1+X_0}=\frac{X_{-1}}{1+X_{-1}}- \frac{X_0}{1+X_0}.$$	
 \end{example}

\section{Shift-plethystic trees}
In this section we deal with shift-plethystic trees in in the context of noncommutative series. This notion was introduced in \cite{MendezJuly2020}, based in the similar construction formalized by Joyal in \cite{joyal1981theorie} and its plethystic generalization in the commutative framework of colored species \cite{Mendezava}.  

Let us consider rooted plane trees whose vertices are colored with colors in $\N$. We associate to a such  tree $T$ the word $\omega(T)=X_{\kap}$, where $\kap$ is the weak composition obtained by reading the vertices of $T$ in preorder. We denote by $\sh^k T,\;k\in \N$, the plane tree obtained by adding $k$ to the color vertex in $T$. It is clear that $\omega(\sh^kT)=\sh^kX_{\kap}.$

 Let $M$ be a non-commutative series in $\K\langle\langle \X\rangle\rangle$, such that $\langle M,1\rangle=1$. We define the noncommutative series of $M$-enriched trees by the implicit equation  
\begin{equation}\label{eq.spltrees}
\msA_M=X_0(M\spl\msA_M).
\end{equation}
 Proposition \ref{prop.implicit} with $F(X,Y)=X_0M(Y)$ assures the existence of a unique solution $\msA_M$. We also obtain that $\msA_{M}$ has as shift-phethystic inverse
 \begin{equation}\label{ec.invtree}
 \msA_M=X_0 M^{-1}
 \end{equation}
 \begin{figure}[hbt!]
 	\begin{center}\includegraphics[width=130mm]{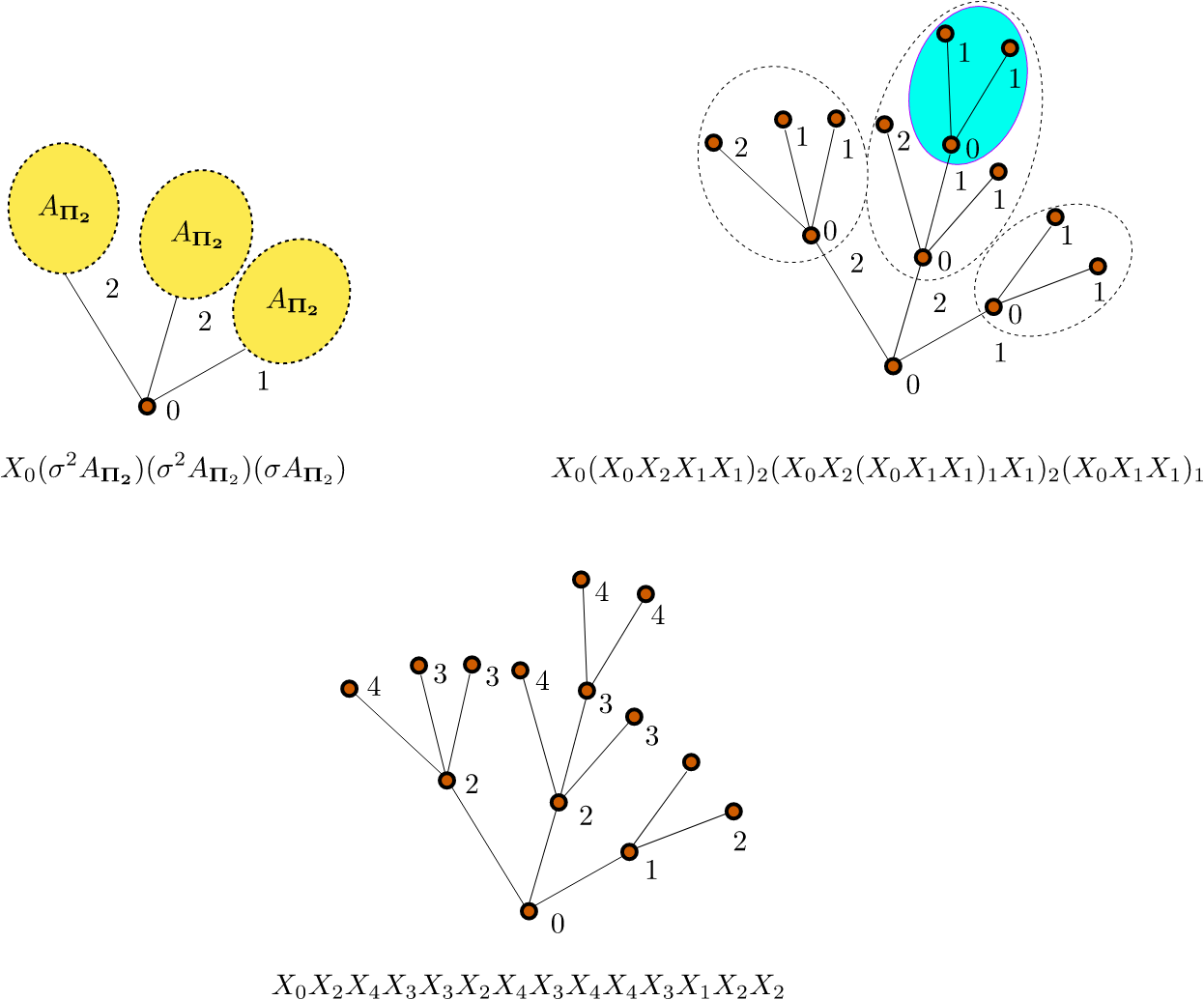}
 	\end{center}\caption{The implicit equation $\msA_{\Pib_{2}}=X_0(\Pib_{2}\spl\msA_{\Pib_{2}})$ defining shift-plethystic trees $\msA_{\Pib_{2}}$ and corresponding associated word.}\label{fig.GeneralizedPlethy}
 \end{figure}
 
 From Eq. (\ref {eq.spltrees}) we obtain the recursion
  \begin{equation}\label{eq.menrichedtree}
 \begin{split}
 \msA_M&=X_0\sum_{\kap}\la M,X_{\kap}\ran(X_{k_1}\spl \msA_{M})(X_{k_2}\spl \msA_{M})\dots (X_{k_{\ell}}\spl \msA_M)\\&=X_0\sum_{\kap}\la M,X_{\kap}\ran (\sh^{k_1} \msA_{M})(\sh^{k_2} \msA_{M})\dots (\sh^{k_{\ell}} \msA_M)
 \end{split}
 \end{equation}
 
 If we assume that $M$ is a language, we get that $\msA_M$ is the series defined by the recursive formula
 \begin{equation}\label{eq.implicitlanguagetree}
\msA_M= \sum_{\kap\in M} X_0(\sh^{k_1} \msA_{M})(\sh^{k_2} \msA_{M})\dots (\sh^{k_{\ell}} \msA_M)
 \end{equation}
 It is the series whose words are associated to a class of colored rooted plane trees that we describe recursively as follows.
  
 \begin{enumerate}
 	\item Its root is colored $0$.
 	\item If the root have $\ell$ children with colors $k_1,k_2,\dots,k_{\ell}$, then $\kap=(k_1,k_2,\dots,k_{\ell})$ is a word in $M$. 
 	\item Denoting by $T_i$ the subtree formed by the descendants of of a vertex whose color $k_i$, then we have that $T_i$ is a  ($k_i$-shifted) shift-plethystic tree; $T_i=\sh^{k_i}T'_i$, $T'_i$ being an $M$-enriched shift-plethystic tree (with root colored zero).
 \end{enumerate}

\begin{remark}\label{rm.spt}\normalfont{It is easy to prove that the trees defined recursively as above are completely described by the properties
	\begin{enumerate} 
		\item Its root is colored zero
		\item For each internal vertex $v$, if its color is $k$ and the word with the colors of its children is $\kap=(k_1,k_2,\dots,k_{\ell})$, then the word $\kap-k=(k_1-k,k_2-k,\dots,k_{\ell}-k)$ is in $M$.
	\end{enumerate}
The shifted trees enumerated by $\sh^n \msA_M$ are similarly described, except that the root has color $n$. 
The general combinatorial description of the series $\msA_M$, when $M$ is not a language, is obtained by weighting the shift-plethystic trees corresponding to the support of $M$. It is done by assigning to each internal vertex $v$ the weight $\la \sh^k M,X_\kap\ran=\la M,X_{\kap-k}\ran$, where $k$ is the color of $v$ and $\kap$ is the word of colors of its children.}
\end{remark}
\begin{example}\label{ex.2partitionenriched} Let $M=\Pib_2=\frac{1}{(1-X_2)(1-X_1)}$ be the language of partitions in weakly decreasing order and using only the parts $1$ and $2$. The shift-plethystic trees corresponding to words of $\msA_{\Pib_2}$ are represented in Fig. \ref{fig.GeneralizedPlethy}. The series $\msA_{\Pib_2}$ is the two-headed hydra continued fraction

 \begin{equation*}\msA_{\Pib_2}=X_0\frac{1}{\left(1-X_2\frac{1}{\left(1-X_4\frac{1}{\iddots\ddots}\right)\left(1-X_3\frac{1}{\iddots\ddots}\right)}\right)\left(1-X_1\frac{1}{\left(1-X_3\frac{1}{\iddots\ddots}\right)\left(1-X_2\frac{1}{\iddots\ddots}\right)}\right)} \end{equation*} 	

We can generalize $\msA_{\Pib_2}$ to shift-plethystic trees enriched with partitions which parts are upper bounded  by $m$, $\msA_{\Pib_{m}}$, leading to the $m$-headed continued fraction. 
\begin{equation}\label{eq.hydramhaedp}\msA_{\Pib_{m}}=X_0\prod_{k=m}^{1}\frac{1}{1-X_k\sh^k\msA_{\Pib_m}}=X_0\frac{1}{\prod_{k_1=m}^{1}1-X_{k_1}\prod_{k_2=m+k_1}^{1+k_1}\frac{1}{1-X_{k_2}\prod_{k_3=m+k_1+k_2}^{1+k_1+k_2}\frac{1}{1-X_{k_3}\cdots}}}.
		\end{equation}
Trees may even be enriched with the language of partitions of any size, $\Pib_{\infty}$, leading to an $\infty$-headed hydra fraction.
\begin{equation*}\msA_{\Pib_{\infty}}=X_0\prod_{k=\infty}^{1}\frac{1}{1-X_k\sh^k\msA_{\Pib_\infty}}=X_0\frac{1}{\prod_{k_1=\infty}^{1}1-X_{k_1}\prod_{k_2=\infty}^{1+k_1}\frac{1}{1-X_{k_2}\prod_{k_3=\infty}^{1+k_1+k_2}\frac{1}{1-X_{k_3}\cdots}}}.
\end{equation*} 
\end{example}
\begin{example}Consider the series $$\Pib_m(-X)=\prod_{k=m}^1\frac{1}{1+X_k}.$$  Let $\RR_{m}$ be the series of $\Pib_{m-1}(-X)$-enriched trees, 
	\begin{equation*}\RR_{m}=\msA_{\Pib_{m}(-X)}=X_0\prod_{k=\infty}^m\frac{1}{1+\sh^k\RR_m}
	\end{equation*} 
	Observe that in this case the enriching series is not a language.   
The series $\RR_m$ is the $m$-headed hydra-continued fraction as in Eq. (\ref{eq.hydramhaedp}) with the obvious change of signs. It generalizes the Rogers-Ramanujan continued fraction. By Eq. (\ref{ec.invtree}), its shift-plethystic inverse is
	$$(\RR_{m})^{\la-1\ran}=X_0(\Pib_{m}(-X))^{-1}=X_0\prod_{k=1}^{m}(1+X_j)=X_0\Pi^m$$
\end{example}
\begin{theorem}\label{theo.quotientpart}The hydra-continued fractions $\RR_{m-1}(X), \RR_{m-1}(\x,z)$ and $\RR_{m-1}(z)$ can be written respectively as the following quotients of  the generating functions of $m$-distinct partitions, 
\begin{eqnarray} &&\label{eq.quotientpart1}\RR_{m-1}=X_0(\sh^{m-1}\Pa_m)(\Pa_m)^{-1}\\
&&\RR_{m-1}(\x,z)=x_0\frac{\Pa_m(zx_{m-1},zx_{m},zx_{m+1},\dots)}{\Pa_m(zx_1,zx_2,zx_3,\dots)}\\
&&\RR_{m-1}(z)=z\frac{\Pa_m(zq^{m-1})}{\Pa_m(z)}.
\end{eqnarray}
\end{theorem}
\begin{proof}
	It is enough to prove Eq. (\ref{eq.quotientpart1}). By taking right shift-plethysm with $\RR_{m-1}=(X_0\Pi^{m-1})^{\la -1\ran}$ in both sides of Eq.  (\ref{eq.shplethymtoone}) we get 
	\begin{equation*}
	\Pa_{m}=\Pi^{\infty}\spl \RR_{m-1}.
	\end{equation*}
	Then
	\begin{align*}X_0(\sh^{m-1}\Pa_m)(\Pa_m)^{-1}&=X_0(\prod_{k=m}^{\infty}(1+X_k)\spl \RR_{m-1})(\prod_{k=1}^{\infty}(1+X_k)\spl\RR_{m-1})^{-1}\\&=X_0(\prod_{k=m}^{\infty}(1+X_k)(\prod_{k=\infty}^{1}\frac{1}{1+X_k}))\spl\RR_{m-1}\\&=X_0(\prod_{k=m-1}^{1}\frac{1}{1+X_k})\spl \RR_{m-1}=\RR_{m-1}.
	\end{align*}
	\end{proof}

Using Eq. (\ref{eq.mdistinct}),
\begin{equation}\label{eq.mrrfraction}
\RR_{m-1}(z)=z\frac{\Pa_{m}(zq^{m-1})}{\Pa_m(z)}=z\frac{\sum_{k=0}^{\infty}\cfrac{q^{m\binom{k+1}{2}}z^k}{(q;q)_k}}{\sum_{k=0}^{\infty}\cfrac{q^{m\binom{k}{2}+k}z^k}{(q;q)_k}}.
\end{equation}

Making $m=2$, from Eq. (\ref{eq.mrrfraction}) we recover the classical Rogers-Ramanujan continued fraction expressed as the ratio of the two Rogers-Ramanujan functions $\Pa_2(zq)$ and $\Pa_2(z)$.

By Eq. (\ref{eq.kdual}), since we have 
$\C^{(m-1)}=\Pa_{m}^!$ (Example \ref{ex.dualpm}), and $\sh^{m-1}\C^{(m-1)}=(\sh^{m-1}\Pa_m)^!$, from Theorem \ref{theo.quotientpart} we obtain
\begin{corollary}\label{cor.quotientcomp}The $(m-1)$-headed continued fraction $\msA_{\Pib_{m-1}}$, its abeleanization and $q$-series can be respectively expressed as the quotients
	\begin{eqnarray}
	&&\label{eq.comptree1}\msA_{\Pib_{m-1}}=X_0\,(\sh^{m-1}\C^{(m-1)})^{-1}(\C^{(m-1)})\label{eq.firstquotient}\\
	&&\msA_{\Pib_{m-1}}(\x,z)=zx_0\frac{\C^{(m-1)}(\x,z)}{\sh^{m-1}\C^{(m-1)}(\x,z)}\\
	&&\msA_{\Pib_{m-1}}(z)=z\frac{\C^{(m-1)}(z)}{\C^{(m-1)}(zq^{m-1})}.
	\end{eqnarray}	
	\end{corollary}	
By umbralization, and using Eq. (\ref{eq.mm1comp}), we get
\begin{equation}\label{eq.hydraminus}\begin{split}
&\msA_{\Pib_{m-1}}(z)=\\&=\frac{z}{\prod_{k_1=1}^{m-1}\left(1-\frac{zq^{k_1}}{\prod_{k_2=1+k_1}^{m+k_1-1}\left(1-\frac{zq^{k_2}}{\prod_{k_3=1+k_1+k_2}^{m+k_1+k_2-1}\left(1-\frac{zq^{k_3}}{\ddots}\right)}\right)}\right)}=\frac{1+\sum_{k=1}^{\infty}(-1)^k\cfrac{q^{m\binom{k}{2}+k}}{(q;q)_k}z^k}{1+\sum_{k=1}^{\infty}(-1)^k\cfrac{q^{m\binom{k+1}{2}}}{(q;q)_k}z^k},\end{split}
\end{equation}
In the following subsection we shall prove Corollary \ref{cor.quotientcomp} by establishing a natural combinatorial link between shift-plethystic trees (enriched with partitions) and cyclic compositions.

\subsection{Compositions as branched shift-plethystic trees.}

We say that a composition $\kap=(k,k_2,\dots,k_{\ell})$ is {\em {\em  cyclic}} if its least component is the first one,
$$k<k_i,\, \mbox{for every $i$, }1<i\leq \ell.$$ 
The word $X_{\kap}$ corresponding to the cyclic composition $\kap$ can be written as
$$X_{\kap}=\sh^{k} X_{0}X_{\kap'},$$ where $\kap'$ is the composition $(k_2-k, k_3-k,\dots, k_{\ell}-k).$  
Then, the generating series of the cyclic compositions having  $k$ as minimum is equal to $\sh^{k} X_0\C.$
A composition can be  uniquely factored as a list of cyclic compositions
\begin{equation}\label{eq.minima}\kap=\mu_1\boldsymbol{\omega}_1|\mu_2\boldsymbol{\omega}_2|\dots |\mu_k\boldsymbol{\omega}_k,\end{equation}
where $\boldsymbol{\mu}=(\mu_1,\mu_2,\dots,\mu_k)$ is a partition (listed in decreasing form). The partition $\boldsymbol{\mu}$, that we call the {\em local minima list}, is defined recursively as follows. We make $\mu_1$ equal to the first element of the composition. Given that we have defined $\mu_1,\mu_2,\dots, \mu_j$, $1\leq j\leq k-1$, define $\mu_{j+1}$ as the first element of $\kap$  after $\mu_{j}$  that is less than or equal to  it. Once we have found $\boldsymbol{\mu}$, put a bar before each $\mu_i$, $i=2, 3,\dots,k$. Each composition $\mu_i\boldsymbol{\omega}_i$ is obviously cyclic and $X_{\kap}$ is in the language,
\begin{equation*}(X_{\mu_1}\spl X_0\C)(X_{\mu_2}\spl X_0\C)\dots (X_{\mu_k}\spl X_0\C)\end{equation*}

\begin{theorem}\label{th.infhydra} We have the identity, 
	\begin{equation}\label{eq.partitioncomposition}\C=\Pib_{\infty}\spl X_0 \C
	\end{equation}
\end{theorem}
\begin{proof}	Denote by $\C_{\boldsymbol {\mu}}$ the generating series of the compositions having partition $\boldsymbol {\mu}=(\mu_1,\mu_2,\dots,\mu_k)$ as local minima list. We have the identity
	\begin{equation*}
	\C_{\boldsymbol {\mu}}=(\sh^{\mu_1}X_0\C)(\sh^{\mu_2}X_0\C)\dots(\sh^{\mu_k}X_0\C)=X_{\boldsymbol {\mu}}\spl X_0\C
	\end{equation*} 
	Since $\C=\sum_{\boldsymbol {\mu}}\C_{\boldsymbol {\mu}}$, we obtain
	\begin{equation*}
	\C=\sum_{\boldsymbol {\mu}}X_{\boldsymbol {\mu}}\spl X_0\C=\prod_{k=\infty}^1 \frac{1}{1-X_k}\spl X_0\C.
	\end{equation*}  
	
\end{proof}
\begin{figure}
	\begin{center}\includegraphics[width=100mm]{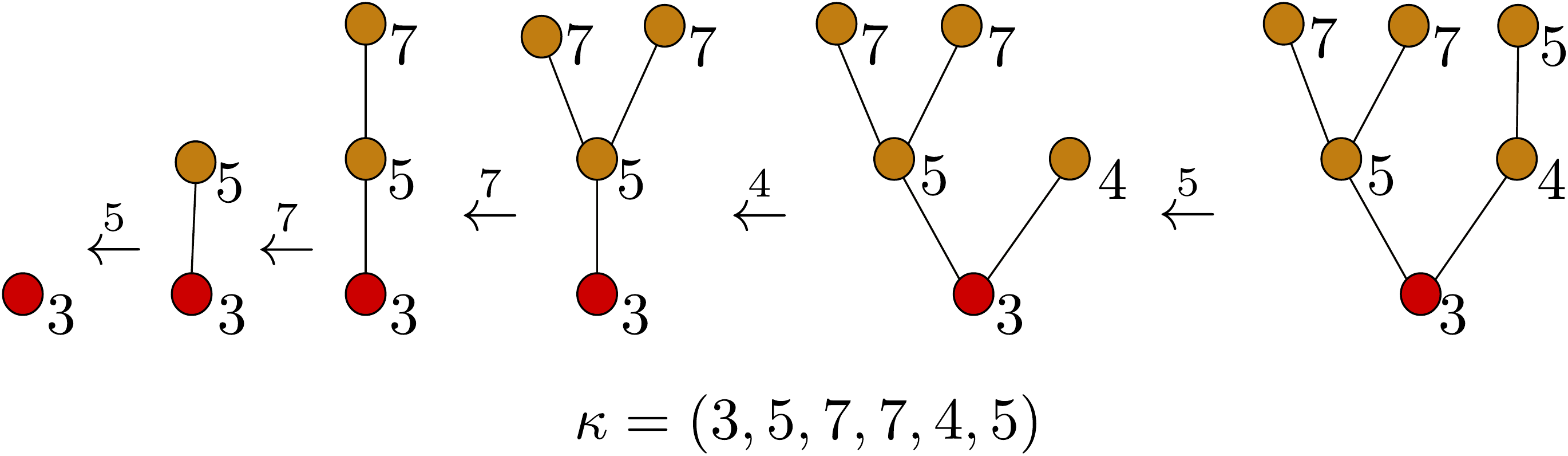}
	\end{center}\caption{Insertion algorithm from $\kap=(3,5,7,7,4,5)$ to the associated shift-plethystic tree in $\sh^3\msA_{\Pib_{\infty}}$ }\label{fig.insertion}
\end{figure}
\begin{corollary}\label{cor.treescompositions}
	We have the identity
	\begin{equation}\label{eq.comptrees}
	X_0\C=\msA_{\Pib_\infty}
	\end{equation}
\end{corollary}
\begin{proof}
	Multiplying by $X_0$ both sides of Eq. (\ref{eq.partitioncomposition}),  we get	that the series  $X_0\C$ satisfy the implicit equation $X_0\C=X_0(\Pib_{\infty}\spl X_0\C)$, which is the same implicit equation defining $\msA_{\Pib_{\infty}}$.\end{proof}
By shifting in Eq. (\ref{eq.comptrees}) we obtain the identity
\begin{equation}\label{eq.shiftcyclic}
X_k\sh^k\C=\sh^k \msA_{\Pib_\infty}
\end{equation}
The implicit equation defining $\msA_{\Pib_{\infty}}$ together with  Eq. (\ref{eq.shiftcyclic}) gives us the following bijection between cyclic compositions and shift-plethystic trees. Consider a cyclic composition $\kap=(k,k_2, k_3,\dots,k_\ell)$, 
\begin{enumerate}\item \label{item.one} Choose the root of the tree to be $k$,
	\item\label{item.two} Factor $(k_2, k_3,\dots,k_\ell)$ as in Eq. (\ref{eq.minima}).
	\item Attach to the root $k$ the local minima list $\mu_1, \mu_2,\dots,\mu_k$ in the same order. \item Apply the same procedure from Item \ref{item.two} to each  of the compositions $\boldsymbol{\omega}_i$, having as root $\mu_i$  for $i=1$ to $k$. Continue until each composition in each branch is a singleton.\end{enumerate}

Alternatively, the following simple algorithm also builds the same shift-plethystic tree out of the composition $\kap$. 
\begin{definition}\textbf{Insertion algorithm}\label{def.insertion}
	\begin{enumerate}
		\item  Given a cyclic composition $\kap=(k,k_2,k_3,\dots,k_{\ell})$, define $T_1$ as the singleton tree with root labeled $k_1=k$.  
		\item {\em {\em  Insertion.}}
		If $\ell\geq 2$, assume that we have constructed a tree $T_j$ with labels $(k,k_2,\dots,k_j)$, the composition with the first $j$ components of $\kap$, $2\leq j<\ell$. To insert the vertex with label $k_{j+1}$ in the tree $T_j$ we follow the steps,
		\begin{enumerate}\item  Look at the rightmost branch of $T_{j}$, from leaf to root, for the first vertex which label is strictly less than $k_{j+1}$. We are sure that there exists such  vertex, because the root $k$ is strictly less than any other component of $\kap.$
			\item\label{stepb} Once we find this vertex,  we append $k_{j+1}$ to it as its rightmost child. It is clear that $k_{j+1}$ is less than or equal to the label of its left hand side sibling, if any. \end{enumerate}
		\item By successively inserting the components of $(k_2,k_3,\dots,k_{\ell})$, we get a tree $T_{\ell}$ (see Fig. \ref{fig.insertion}).
	\end{enumerate}
\end{definition}

\begin{proposition}\label{prop.insertion}\normalfont{The insertion algorithm establishes a bijection between cyclic compositions with first element equal to $k$, and the words in $\sh^k\msA_{\Pib_{\infty}}$ associated to shift-plethystic trees.} 
\end{proposition}
\begin{proof}	
	By induction we can see that the vertex labeled $k_j$ in $T_j$ is the leaf in its rightmost branch. So, we can recover the composition by taking out each time the rightmost leaf, putting them in a list, and then reading it backwards. The procedure to obtain the first list is called in the literature {\em post order} (from right to left). Reading this list backwards is equivalent to read the vertices of $T_\ell$ in preorder. 
	
	The word $\omega(T_{\ell})$ is in $\sh^k\msA_{\Pib_{\infty}}$,  because its root is $k$, and by Step \ref{stepb} of the algorithm, its  vertices are in strictly increasing order from father to son, and the children of each vertex are in weakly decreasing order from left to right. Equivalently, if $v$ is an internal vertex colored $k$ and the list of the colors of its children is $\kap$, then $(k_1-k, k_2-k,\dots,k_\ell-k)$ is a word in $\Pib_{\infty}$ according to Remark \ref{rm.spt}.\end{proof}
\begin{remark}\normalfont{Since $k_j$ is the leaf in the rightmost branch of $T_j$, if $j$ is a rise ($k_{j+1}>k_j$), the $j$th insertion will append $k_{j+1}$ as the first child of $k_j$. In the final tree $T_{\ell}$, $k_{j+1}$ will be the leftmost child of $k_j$, and its number of internal vertices give us the number of rises in $\kap$.}
\end{remark}
The same algorithm can be applied to cyclic compositions whose contiguous differences are upper bounded by some positive integer number $m$. 
\begin{theorem}\label{theo.mcompositionsfactor}
	The insertion algorithm gives a bijection between cyclic compositions in $\C^{(m)}$ (with first component equal to $k$), and words corresponding to shift-plethystic trees in $\sh^k\msA_{\Pib_{m}}$. Moreover we have the identity
	\begin{equation}\label{eq.mfactor}
	\C^{(m)}=\Pib_{\infty}\spl \msA_{\Pib_m}
	\end{equation}
\end{theorem}
\begin{proof} Let $v$ be an internal vertex of tree associated to a tree $T_{\ell}$ coming from a cyclic composition in $\C^{(m)}$. If the color of $v$ is equal to $k_j$ for some part $k_j$ of $\kap$, then, its leftmost child has label $k_{j+1}$. Since $k_{j+1}-k_{j}\leq m$, then for any other child of $v$ with color say $k_r$, since $k_j< k_{r}\leq k_{j+1}$ we have that $1\leq k_r-k_j\leq k_{j+1}-k_j\leq m$. Hence, by Remark \ref{rm.spt}, the word of $T_{\ell}$ is in $\sh^k\msA_{\Pib_{m}}$.\end{proof} 
We are now ready to give an alternative proof of Corollary \ref{cor.quotientcomp}.
\begin{proof}
It is enough to prove Eq. (\ref{eq.comptree1}). 	 By Eq. (\ref{eq.mfactor}) we have that 
	\begin{equation*}\begin{split}
	X_0(\sh^{m-1} \C^{(m-1)})^{-1}(\C^{(m-1)})&=X_0(\prod_{k=\infty}^{m}\frac{1}{1-X_k}\spl\msA_{\Pib_{m-1}})^{-1}(\prod_{k=\infty}^1\frac{1}{1-X_k}\spl \msA_{\Pib_{m-1}})\\&=X_0(\prod_{k=m}^\infty(1-X_k)\prod_{k=\infty}^1\frac{1}{(1-X_k)})\spl \msA_{\Pib_{m-1}}\\&=X_0(\prod_{k=m-1}^1\frac{1}{1-X_k})\spl\msA_{\Pib_{m-1}}=\msA_{\Pib_{m-1}}.
	\end{split}
	\end{equation*}
\end{proof}

\begin{remark} Theorem \ref{th.infhydra} gives us also an expansion for $\C$ as a product of  $\infty$-headed hydra fraction  $$\C=\prod_{k=\infty}^{1}\frac{1}{1-X_k\sh^k\C}=\prod_{k_1=\infty}^{1}\frac{1}{1-X_{k_1}\prod_{k_2=\infty}^{1+k_1}\frac{1}{1-X_{k_2}\prod_{k_3=\infty}^{1+k_1+k_2}\frac{1}{1-X_{k_3}\cdots}}}.$$
	and its associated $q$-series
	$$\C(z)=\prod_{k=1}^{\infty}\frac{1}{1-q^k\C(zq^k)}=\prod_{k_1=1}^{\infty}\frac{1}{1-\prod_{k_2=1+k_1}^{\infty}\frac{zq^{k_1}}{1-\prod_{k_3=1+k_1+k_2}^{\infty}\frac{zq^{k_2}}{1-zq^{k_3}\cdots}}}.$$\end{remark}
Eq. (\ref{eq.partitioncomposition}) can now be rewritten as
$$\C=\Pib_{\infty}\spl\msA_{\Pib_{\infty}}.$$

From that we obtain a representation of compositions in terms of ordered forests of shift-pletystic trees (see Fig. \ref{compositiondecomposition}). 
\begin{figure}
	\begin{center}\includegraphics[width=80mm]{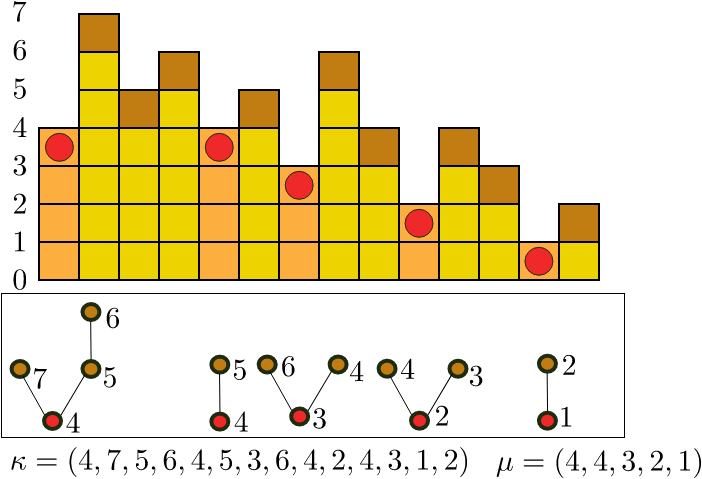}
	\end{center}\caption{Composition $\kap$, associated local minima partition $\boldsymbol {\mu}$, and ordered forest of shift-plethystic trees. }\label{compositiondecomposition}
\end{figure}

\begin{theorem}\label{theo.complocalminima}\normalfont{
		The abeleanization of the generating function of the compositions and its associated $q$-series, with the power of $t$ indicating the number of local minima, are given respectively by
		\begin{eqnarray}\label{eq.compominima}
		\C(\x,t)&=&\prod_{k=1}^{\infty}\frac{1-\sum_{j={k+1}}^\infty x_j}{1-\sum_{j=k+1}x_j-tx_k}\\\label{eq.compominima2}
		\C(z,t)&=&\prod_{k=1}^\infty\frac{1-q-zq^{k+1}}{1-q-zq^{k+1}+q^k(q-1)zt}
		\end{eqnarray}}
\end{theorem}
\begin{proof}
	The abelianization of $X_0\C$ is equal to 
	$\ab(X_0\C)=x_0\C(\x)=\frac{x_0}{1-\Al_1}$. By Eq. (\ref{eq.partitioncomposition}), and multiplying $x_0\C(\x)$ by $t$ to keep track of the number of local minima,  we obtain
	\begin{equation*}
	\C(\x)=\prod_{k=1}^{\infty}\frac{1}{1-t\frac{x_k}{1-\Al_{k+1}}}=\prod_{k=1}^{\infty}\frac{1-\Al_{k+1}}{1-\Al_{k+1}-tx_k}.
	\end{equation*} 
	The umbral map $\um:x_j\mapsto zq^j$ give us Eq. (\ref{eq.compominima2}), since
	$\mathfrak{u}(\Al_k)(z)=z\sum_{j=k}^{\infty} q^j=z\frac{q^{k}}{1-q}.$
\end{proof}

\noindent




\subsection{Partitions as branchless shift-plethystic trees}
shift-plethystic trees enriched with `one letter languages' are called branchless. They are obtained by choosing a subset $S$ of $\N$
and enriching with the language $1+\Al_S$. Hence, words in $\msA_{(1+\Al_S)}$ are of the form $X_{\boldsymbol{\lam}}$, where $\boldsymbol{\lam}$ satisfies $0=\lam_1\leq\lam_2\leq \dots\leq\lam_{\ell}$, with all its  rises $\lam_{i+1}-\lam_i$ in $S$. The shift $\sh^i$ applied to $\msA_{(1+\Al_S)}$ will give us the language of partitions (in increase order) with first part equal to $i$, $\lam_1=i$, and rises in $S$. By shift-plethysm with $1+\Al_1=\sum_{k=1}^{\infty}X_k$, we obtain the whole language of partitions with rises in the set $S$ (including the empty one). For example, by enriching with $M=(1+\Al_{\mathrm{odd}}),$ we get the branchless trees with odd rises (see Fig. \ref{fig.branchless}). By shift-plethysm with $(1+\Al_1)$, we get the series of partitions with odd rises,
$$\pa_{\mathrm{odd}}=(1+\Al_1)\spl \msA_{(1+\Al_{\mathrm{odd}})}.$$
	\begin{figure}[hbt!]\label{fig.branchless}
		\begin{center}\includegraphics[width=40mm]{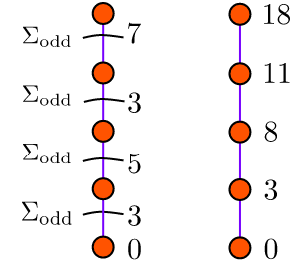}
			\caption{An example of a branchless shift-plethystic tree associated to partitions with odd rises, $\mathscr{A}_{(1+\Al_{\mathrm{odd}})}=X_0((1+\Al_{\mathrm{odd}})\spl\mathscr{A}_{(1+\Al_{\mathrm{odd}})} )$.}
		\end{center}
	\end{figure}

Going back to the general case, we have
		
	\begin{equation*}\begin{split}&\msA_{(1+\Al_S)}=X_0\times(1+\Al_S)\spl \msA_{(1+\Al_S)}, \mbox{ and, }\\&\Pa_S=1+\Al_1\spl \msA_{(1+\Al_S)}.\end{split}
	\end{equation*} 
\noindent The series $\Pa_S^+=\Al_1\spl \msA_{(1+\Al_S)}$ has as shift-plethystic inverse $\msA_{(1+\Al_S)}^{\la -1\ran}\spl (\Al_1)^{\la -1\ran}$, which by Eq. (\ref{ec.invtree}) and Example \ref{ex.inversesigma}, is equal to

	\begin{equation}\label{eq.inversepartition}
	(\pa_S^+)^{\langle-1\rangle}=X_0\frac{1}{1+\sum_{s\in S}X_s}\spl (X_{-1}-X_0)=(X_{-1}-X_0)\frac{1}{1+\sum_{s\in S}(X_{s-1}-X_s)}
	\end{equation}
By substitution of $\pa_S^+$ in the rightmost term of Eq. (\ref{eq.inversepartition})
	\begin{equation}\label{eq.psubs0}
(\sh^{-1}\pa_S^+-\pa_S^+)\frac{1}{1+\sum_{s\in S}(\sh^{s-1}\pa_S^+-\sh^s\pa_S)}=X_0.
\end{equation}
Hence 
\begin{equation}\label{eq.psubs}
\sh^{-1}\pa_S^+-\pa_S^+=X_0(1+\sum_{s\in S}(\sh^{s-1}\pa_S^+-\sh^s\pa^+_S)).
\end{equation}
\begin{theorem}\label{teo.brnchlespart}The $q$-series of $\pa_S$ is given by
\begin{equation}\label{eq.marvel}
\pa_S(z)=1+\sum_{k=1}^{\infty}\frac{q^{\binom{k+1}{2}}}{(1-q^k)}(\Se_{k-1})!(q)z^k
\end{equation}
\noindent Where $$\Se_k(q)=\sum_{s\in S}q^{k(s-1)},$$ and $(\Se_{k})!(q)$ is a symbolic expression for $\Se_k(q)\Se_{k-1}(q)\Se_{k-2}(q)\dots\Se_1(q),$
\begin{equation*}(\Se_{k})!(q):=\Se_k(q)\Se_{k-1}(q)\Se_{k-2}(q)\dots\Se_1(q).\end{equation*} 
\end{theorem}
\begin{proof}Let us denote by $f_k^S(q)$ the coefficients in the expansion of $\pa_S^+(z)$,
$$	\pa^+_S(z)=\sum_{k=1}^{\infty}f_k^S(q)z^n$$
The umbral morphism applied to the shift-plethystic inverse in Eq. (\ref{eq.psubs}) gives us
\begin{equation}\label{eq.psubs1}
\pa^+_S(z/q)-\pa^+_S(z)=z(1+\sum_{s\in S}(\pa_S^+(zq^{s-1})-\pa_S^+(zq^{s})))
\end{equation}
Taking the coefficient of $z^k$, $k\geq 1$, in both sides of Eq. (\ref{eq.psubs1}) we obtain the formula
$$q^{-k}f^S_{k}(q)-f_k^S(q)=\sum_{s\in S}(q^{(s-1)(k-1)}-q^{s(k-1)})f_{k-1}^S(q)=(1-q^{k-1})(\sum_{s\in S}q^{(s-1)(k-1)})f_{k-1}^S(q),$$
with initial condition $f_1^{S}(q)=1$. From that we obtain the recursion
$$f_k^S(q)=\frac{q^k}{1-q^k}(1-q^{k-1})(\sum_{s\in S}q^{(s-1)(k-1)})f_{k-1}^S(q)=\frac{q^k}{1-q^k}(1-q^{k-1})\mathcal{S}_{k-1}(q)f_{k-1}^S(q),$$
which gives us the result,
\begin{equation}\label{eq.marv1}
f_k^S(q)=\frac{q^{\binom{k+1}{2}}(q,q)_{k-1}}{(q,q)_k}(\Se_{k-1})!(q)=\frac{q^{\binom{k+1}{2}}}{1-q^k}(\Se_{k-1})!(q).
\end{equation}

\end{proof}

				

\begin{example}Consider the case of $m$-distinct partitions. The set of risings is equal to the integer interval $S=[m,\infty)=\{m, m+1, m+2, \dots\}$ and we have 
$$\mathcal{S}_k(q)=\sum_{s=m}^{\infty}q^{(s-1)k}=\sum_{s=m-1}^{\infty}(q^k)^s=\frac{q^{k(m-1)}}{1-q^k}.$$
Then, 
$$(\Se_{k-1})!(q)=\frac{q^{(m-1)\binom{k}{2}}}{(q;q)_{k-1}},$$
from Eq. (\ref{eq.marv1}) we get,
$$f_k^S(q)=\cfrac{q^{\binom{k+1}{2}+(m-1)\binom{k}{2}}}{(q;q)_k}=
\cfrac{q^{m\binom{k}{2}+k}}{(q;q)_k}
$$
and we recover Eq. (\ref{eq.mdistinct}).
\end{example}
\begin{example} The rises are in the integer interval  $S=[m,n]=\{k\in\N|m\leq k\leq n\}$, $0\leq m\leq n.$
	In this case $$\Se_k(q)=\sum_{s=m-1}^{n-1}q^{ks}=q^{k(m-1)}\frac{1-q^{k(n-m+1)}}{1-q^k},$$
	and $$\Se_{k-1}!(q)=\frac{q^{(m-1)\binom{k}{2}}(q^{n-m+1};q^{n-m+1})_{k-1}}{(q;q)_{k-1}}$$
	
	  By Eq. (\ref{eq.marv1})
$$f_k^S(q)=\frac{q^{m\binom{k}{2}+k}(q^{n-m+1};q^{n-m+1})_{k-1}}{(q;q)_k}$$
and
\begin{equation*}
\pa_{[m,n]}(z)=1+\sum_{k=1}^\infty \frac{q^{m\binom{k}{2}+k}(q^{n-m+1};q^{n-m+1})_{k-1}}{(q;q)_k}z^k.
\end{equation*} 
In particular, for $m=1$,
\begin{equation*}
\pa_{[1,n]}(z)=1+\sum_{k=1}^\infty \frac{q^{\binom{k+1}{2}}(q^n;q^n)_{k-1}}{(q;q)_k}z^k.
\end{equation*} 
And for $m=n$, $S=\{m\}$,
\begin{equation*}
\pa_{\{m\}}(z)=1+\sum_{k=1}^\infty \frac{q^{m\binom{k}{2}+k}}{1-q^k}z^k.
\end{equation*} 

\end{example}
\begin{example}The rises are multiples of $m$, $S=m\N$, $m\in \N$. In this case $$\Se_k(q)=\sum_{j=0}^\infty q^{k(mj-1)}=\frac{1}{q^k(1-q^{mk})}.$$
	$(\Se_{k-1})!(q)$ is easily computed,
	$$(\Se_{k-1})!(q)=\frac{1}{q^{\binom{k}{2}}(q^m;q^m)_{k-1}}$$
	 Then we get 
	\begin{equation}\label{eq.multiple1}\pa_{m\N}(z)=1+\sum_{k=1}^\infty \frac{q^k}{(1-q^k)(q^m;q^m)_{k-1}}z^k.\end{equation}
	If we exclude zero from the set of rises, $S=m\N_+$, we get the formula for the set of partitions with risings a multiple of $m$, without repetitions
		\begin{equation}\label{eq.multiple2}\pa_{m\N_+}(z)=1+\sum_{k=1}^\infty \frac{q^{m\binom{k}{2}+k}}{(1-q^k)(q^m;q^m)_{k-1}}z^k.\end{equation}
		Since in this case $$\Se_k(q)=\frac{q^{mk}}{q^k(1-q^{mk})}=\frac{q^{(m-1)k}}{1-q^{mk}}.$$
	
\end{example}
\begin{example}The rises are congruent with $l$ module $m$, $0\leq l\leq m$. Then $$\Se_k(q)=\frac{q^{(l-1)k}}{1-q^{mk}},$$ and we obtain
		\begin{equation*}
		\pa_S(z)=1+\sum_{k=1}^\infty \frac{q^{l\binom{k}{2}+k}}{(1-q^k)(q^m;q^m)_{k-1}}z^k
		\end{equation*} 
\noindent from which we recover, for $l=0$ and $l=m$, formulas (\ref{eq.multiple1}) and (\ref{eq.multiple2}) respectively. 
In particular, for $m=2$ and $l=1$, we obtain the generating function of partitions with odd rises.
\begin{equation*}
\Pa_{\mathrm{odd}}(z)=1+\sum_{k=1}^\infty \frac{q^{\binom{k+1}{2}}}{(1-q^k)(q^2;q^2)}z^k.
\end{equation*} 	
\end{example}

Observe that Theorem \ref{teo.brnchlespart} has a dual version in terms of compositions. The language $\Pa_S$ is linked, having as set of links $B_S=\{(i,j)|j-i\in S\}\subseteq \N_+\times\N_+$. Its dual language is that of compositions $\kap$ such that $k_{i+1}-k_i\notin S$. More precisely, denoting by $\widehat{S}$ the complement of $S$ in $\Z$, and by $\C^{\widehat{S}}$  the language of compositions $\kap$ such that $k_{i+1}-k_i\in \widehat {S}$, $i=1,2,\dots,\ell(\kap)-1$, we have 
\begin{equation}\label{eq.dualityparticom}\Pa_S^!=\C^{\widehat{S}}.
\end{equation}
For example, $\Pa_{\mathrm{odd}}^!$ is the language of compositions such that $k_{i+1}-k_i$ is either even and non negative, or negative.
 Using Eq. (\ref{eq.dualityparticom}), and Proposition \ref{prop.kdual1}, from Theorem \ref{teo.brnchlespart} we get 
\begin{corollary}\label{cor.compositionsSc} The $q$-series of $\C^{\widehat{S}}$ is given by
	\begin{equation}\label{eq.marvel1}
	\C^{\widehat{S}}(z)=\left(1+\sum_{k=1}^{\infty}(-1)^k\frac{q^{\binom{k+1}{2}}}{(1-q^k)}(\Se_{k-1})!(q)z^k\right)^{-1}.
	\end{equation} 
\end{corollary}
\section{Appendix}
\subsection{Proof of Proposition \ref{prop.implicit}}\label{sec.appendix}
\begin{proof}Before giving the details of the proof, we begin with an example. Let us assume that $F(X;Y)$ is a language, and that for example $\boldsymbol{\tau}=X_2\textcolor{red}{Y_2}X_0\textcolor{red}{Y_0}$ is a word in $F(X;Y)$. Then, the substitution of $G^{(1)}(X)=F(X;0)$ into $\boldsymbol{\tau}$ is equal to $X_2\sh^2G^{(1)}(X)X_0G^{(1)}(X)$. If the words $X_0X_3$ and $X_1X_2$ are in the language $G^{(1)}$, then $X_2X_2X_5X_0X_1X_2$ is in $G^{(2)}$. This can be represented as a tree with height two and two kinds of colored edges. Colors of the edges are two kind of integers in $\N$, `red or black' depending on the letter of the word being in $\X$ or in $\Y$. Leaves and edges pointing to leaves are colored only with `black numbers'.  Here, the {\em  height of a tree} is defined to be maximun of the heights of its leaves, the {\em  height of a leaf} being the number of internal vertices in the path from the root. Reading from left to right the letters corresponding to the (black) leaves  we obtain {\em  the word $\Omega(T)$} in the support of $G^{(2)}=F(X;F(X;0))$. As a matter of fact, the series $G^{(2)}$ is obtained by adding the words associated to these kind of trees, with height at most $2$ (see Fig. \ref{fig.Gtree1}). Those of height $1$ are identified with words in $G^{(1)}(X)=F(X;0)=F(X,F(0;0))$, since $F(0,0)=0$. It is not difficult to prove  by induction  that the series $G^{(n)}$ is obtained  by adding the words associated to trees with height at most $n$, enriched with words in $F(X;Y)$ (trees with bi-colored edges). Then we have
\begin{equation}\label{eq.implcittrees}G^{(n)}=\sum_{T:\,\mathrm{height}(T)\leq n}\Omega(T).
\end{equation}
With this combinatorial representation of $G^{(n)}$ in mind we can now begin the proof of Proposition \ref{prop.implicit}. We assume, without loss of generality, that $F(X,Y)$ is a language. First we have to prove that $G^{(n)}$ is convergent. Given a word $\boldsymbol{\tau}$ and a tree such that $\Omega(T)=X_{\boldsymbol{\tau}}$, we claim that 
{\em for every component $\tau_i$ of $\boldsymbol{\tau}$, the height of the leaf $v_i$ of $T$ colored  $\tau_i$, is upper bounded by $\tau_i+\ell$, $\ell=\ell(\boldsymbol{\tau})$}. Proof of the claim: let $P_i$ be the path from the root to $v_i$ and by $p_i$ the father of $v_i$, the last internal vertex of $P_i$. Denote by $I_1$ the set of internal vertices in $P_i$ different from $p_i$ and having only one child, and by $I_2$ the set having the rest of internal vertices in $P_i$. Observe that all the edges in $P_i$ are colored red, except the last one, connecting $p_i$ with $v_i$.
Since $\la F(X;Y),Y_0\ran =0$, an edge connecting a vertex in $I_1$ with its child have to be colored \textcolor{red}{red  $k$}, for some $k\geq 1$. It means a shifting of at least one for each of these $r:=|I_1|$ internal vertices, and all these shifts have necessarily to add up at most $\tau_i$. Hence $r\leq \tau_i$. For each vertex $v$ in $I_2$ there is at least one path $P_v$ from $v$ to a leaf in $T$. For $v=p_i$ the path is defined as $\{p_i,v_i\}$. If $v\neq p_i$ it is obtained by choosing a child of $v$ not in $P_i$ (this is because $v$ has at least two children), and then any path going trough this child to a leaf. For $v\neq v'$, both in $I_2$, the leaf in $P_v$ is different from the leaf in $P_{v'}$. Otherwise $T$ would have a cycle, because $v$ and $v'$ are connected trough $P_i$. Since $T$ has a total of $\ell$ leaves, $s:=|I_2|\leq\ell$. Since the height of $v_i$ is equal to $r+s$, and $r+s\leq \tau_i+\ell$, we have proved the claim. 

Once we have proved the claim we have that if $\Omega(T)=X_{\tau}$, then the height of $T$ is upper bounded by $m=\mathrm{max}\{\tau_i+\ell|i=1,2,\dots,\ell\}=\mathrm{max}\{\tau_i|i=1,2,\dots,\ell\}+\ell$. Then,  since $\la \Omega(T), X_{\tau}\ran=0 $ if $\mathrm{heitght}(T)>m$, by Eq. (\ref{eq.implcittrees}), \begin{equation*}\la G^{(n)}, X_{\tau}\ran=\sum_{T:\,\mathrm{height}(T)\leq n}\la \Omega(T), X_{\tau}\ran=\sum_{T:\,\mathrm{height}(T)\leq m}\la \Omega(T), X_{\tau}\ran=\la G^{(m)}, X_{\tau}\ran.\end{equation*}
\begin{figure}[hbt!]\label{fig.Gtree1}
	\begin{center}\includegraphics[width=8cm]{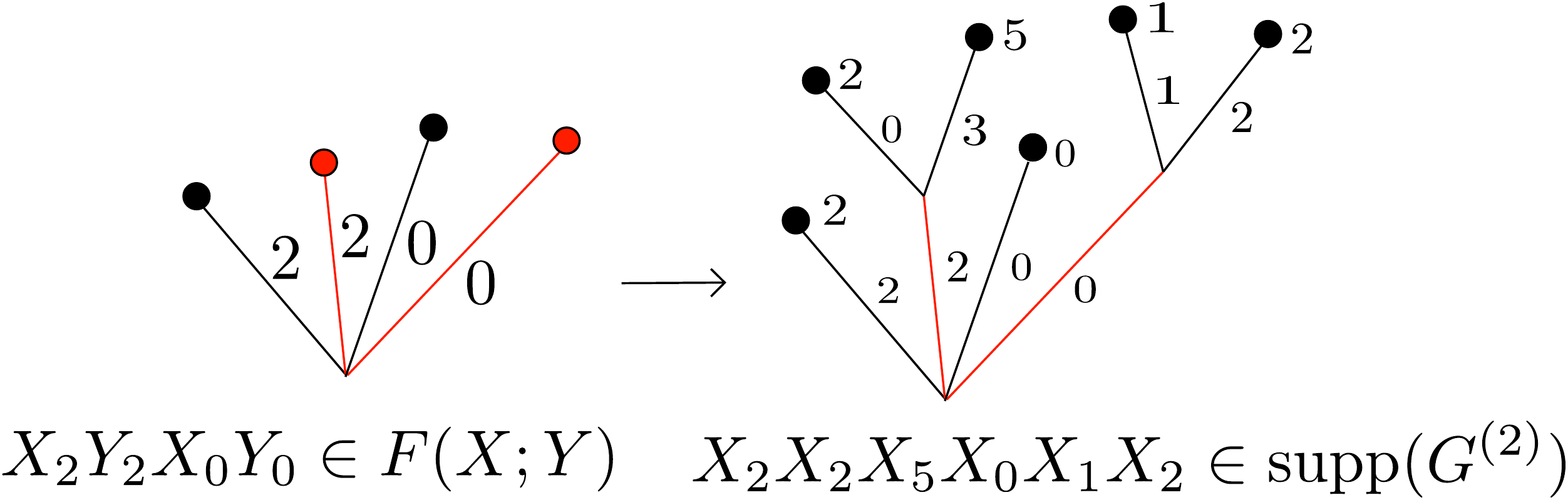}
		\caption{Combinatorial representation of $G^{(2)}=F(X;G^{(1)})$.}
	\end{center}
\end{figure}
Hence, for every $\boldsymbol{\tau}$, the sequence $\la G^{(n)}, X_{\boldsymbol{\tau}}\ran$ is stationary, $G^{(n)}$ converges, and $G=\lim_{n\rightarrow \infty}G^{(n)}$ is a solution of the shift-plethystic implicit equation. To prove unicity we have to introduce some notation. Let $R(X)$ and $S(X)$ be two series with zero constant term. We say the $R=_n S$ if for every word $X_{\boldsymbol{\tau}}$ with $\mathrm{max}\{\tau_i|i=1,2,\dots,\ell(\boldsymbol{\tau})\}+\ell(\boldsymbol{\tau})\leq n$ we have $\la R,X_{\boldsymbol{\tau}}\ran=\la S,X_{\boldsymbol{\tau}}\ran$. Unicity is obtained from the easy implication
$$R=_n S\Rightarrow F(X;R(X))=_{n+1}F(X;S(X))$$
and the fact that if $H$ is another solution, then $H=_0 G$ (because $H(0)=G(0)=0$).\end{proof}

\clearpage
\subsection{Table with notation}\label{table.notation}
\begin{table}[h!]
	\begin{center}
		\caption{Table of symbols for relevant series.}
		\label{tab:table1}
		\begin{tabular}{lcl}
			\toprule 
			\textbf{Symbol} & \textbf{NC Series} & \textbf{Combinatorial meaning}\\\midrule\midrule
			\textcolor{magenta} {$\msA_{M}$}&$\msA_M=X_0(M\spl \msA_M$)& \textcolor{magenta}{Shift-plethystic trees enriched with $M$}.\\\midrule
			$\msA_{\Pib_m}$	& ----- & Shift-plethystic trees enriched with partitions in $\Pib_m$.\\\midrule
		$\msA_{(1+\Al_S)}$	& ----- &Branchless shift-plethystic trees with rises in $S$.\\\midrule
					 \textcolor{magenta}{$\C$}&$\frac{1}{1-\Al_1}=\frac{1}{1-\sum_{k=1}^{\infty}X_k}$&\textcolor{magenta} {Compositions}.\\\midrule 
			 $\mathrm{C}$&$\frac{1}{1-\sum_{i=1}^\infty\frac{X_i}{1+X_i}}.$& Carlitz compositions (no contiguous repeated letters).\\\midrule 
			$\C^S$&-----&Compositions with differences in $S$, $k_{i+1}-k_i\in S$.\\\midrule
			$\C^{(m)}$&-----&Compositions with differences  $k_{i+1}-k_i\leq m$.\\\midrule 
		\textcolor{magenta}{$\Pib$}	&-----&\textcolor{magenta}{Partitions with repetitions, in decreasing order}.\\\midrule
			$\Pib_m$&$\prod_{k=m}^1\frac{1}{1-X_k}$ &  Longest part $\leq m$.\\\midrule
			$\Pib_{\infty}$ & $\prod_{k=\infty}^1\frac{1}{1-X_k}$ &Unbounded size. \\\midrule\textcolor{magenta}{$\Pi$}&-----&\textcolor{magenta}{Distinct partitions, in increasing order}.\\\midrule 
			 $\Pi^m$&$\prod_{k=1}^m(1+X_k)$&Longest part $\leq m$.\\\midrule 
			$\Pi^{\infty}$&$\prod_{k=1}^{\infty}(1+X_k)$&Unbounded size.\\\midrule\textcolor{magenta}{$\Pa$}&-----&\textcolor{magenta}{Partitions, according with their rises}.\\\midrule
			$\Pa_S$&-----&With risings $\lam_{i+1}-\lam_i\in S$.\\\midrule 
			$\Pa_m$&-----& With risings $\lam_{i+1}-\lam_i\geq m$ ($m$-distinct).\\\midrule
			\textcolor{magenta}{$\Al$}&-----&\textcolor{magenta}{Alphabet $\subseteq\N$}\\\midrule
			$\Al_S$ & $\sum_{k\in S}X_k$ & -----\\\midrule 
			$\Al_m$ & $\sum_{k=m}^{\infty}X_k$& -----\\ 
			\bottomrule 
		\end{tabular}
	\end{center}
\end{table}

\bibliographystyle{amsplain}.
\bibliography{bibliodehydra}
\end{document}